\documentclass{amsart}

\usepackage{amssymb,amsmath,amscd,amsthm,enumerate}
\usepackage{color}

\date{\today}

\newcommand{\Z}{{\mathbb Z}}
\newcommand{\R}{{\mathbb R}}

\newcommand{\Leb}{{\mathrm{Leb}}}

\newcommand{\LP}{{\mathrm{LP}}}

\newtheorem{theorem}{Theorem}[section]
\newtheorem{lemma}[theorem]{Lemma}
\newtheorem{prop}[theorem]{Proposition}
\newtheorem{coro}[theorem]{Corollary}

\theoremstyle{definition}
\newtheorem{remark}[theorem]{Remark}

\theoremstyle{definition}

\theoremstyle{definition}

\theoremstyle{definition}

\theoremstyle{definition}

\allowdisplaybreaks

\numberwithin{equation}{section}

\newcommand{\set}[1]{\left\{#1\right\}}

\begin{document}

\title[Multidimensional Operators with Cantor Spectrum]{Multidimensional Almost-Periodic Schr\"odinger Operators with Cantor Spectrum}

\begin{abstract}
We construct multidimensional almost-periodic Schr\"odinger operators whose spectrum has zero lower box counting dimension. In particular, the spectrum in these cases is a generalized Cantor set of zero Lebesgue measure.
\end{abstract}

\author[D.\ Damanik]{David Damanik}
\address{Department of Mathematics, Rice University, Houston, TX~77005, USA}
\email{damanik@rice.edu}
\thanks{D.D.\ was supported in part by NSF grant DMS--1700131.}

\author[J.\ Fillman]{Jake Fillman}
\address{Department of Mathematics, Virginia Tech, 225 Stanger Street, Blacksburg, VA~24061, USA}
\email{fillman@vt.edu}
\thanks{J.F.\ was supported in part by an AMS-Simons travel grant, 2016--2018}

\author[A.\ Gorodetski]{Anton Gorodetski}
\address{Department of Mathematics, University of California, Irvine, CA~92697, USA}
\email{asgor@math.uci.edu}
\thanks{A.G.\ was supported in part by Simons Fellowship (grant number 556910).}

\maketitle



\textbf{MSC2010 Subject Class:} 34L40

\section{Introduction}

A Schr\"odinger operator in $\R^d$ is a self-adjoint operator of the form
\[
L_V \phi
=
-\Delta \phi + V\phi
\]
in $L^2(\R^d)$, where $V:\R^d \to \R$ is a bounded, continuous function, known as the \emph{potential}. One says that $V$ is \emph{periodic} if there are linearly independent $u_1,\ldots,u_d \in \R^d$ with $V(x)= V(x+u_j)$ for all $x$ and $j$, (uniformly) \emph{almost-periodic} if
\[
\set{V(\cdot - y) : y \in \R^d}
\]
is precompact in $C(\R^d)$ with the uniform topology, and \emph{limit-periodic} (denoted $V \in \LP(\R^d)$) if there are periodic $V^{(1)},V^{(2)},\ldots$ in $C(\R^d)$ with $V^{(n)} \to V$ uniformly. It is not hard to check that every limit-periodic potential is also almost-periodic.

It is well known that in the one-dimensional case, $d = 1$, the spectrum of $L_V$ has a tendency to be a Cantor set when  $V \in \LP(\R)$; compare for example, \cite{AS81, DFL2017, FL, MC, M, PT84, PT88}. Moreover, these Cantor sets can be quite thin, in the sense that they may have zero Lebesgue measure; using ideas from Avila's work in \cite{avila2009CMP}, Damanik, Fillman, and Lukic constructed examples of $V \in \LP(\R)$ so that $\sigma(L_V)$ is a zero-measure Cantor set \cite{DFL2017}. More precisely, the zero-measure phenomenon is generic in $\LP(\R)$, whereas for a dense set of $V \in \LP(\R)$, $\sigma(L_V)$ is even of zero Hausdorff dimension \cite{DFL2017}.

When one of the authors of \cite{DFG2014} and \cite{DFL2017} presented the results from those papers at the ICMP 2018 in Montr\'eal, the following question was asked at the end of the talk: is this a purely one-dimensional phenomenon or is there any hope to construct examples of this kind in higher dimensions as well?

The natural inclination is to expect that this is indeed a one-dimensional phenomenon that has no counterpart in higher dimensions. The main reason is given by the work done on the Bethe-Sommerfeld conjecture. If $d = 1$ and, for example, $V(x) = \cos x$, then it is known that $\sigma(L_V)$ is a disjoint union of infinitely many compact intervals \cite[pp.~298--299]{RS78}. In other words, the spectrum has infinitely many gaps at arbitrarily high energies. This topological structure is known to be generic among all one-dimensional periodic potentials \cite{S76}. On the other hand, the \emph{Bethe--Sommerfeld conjecture} posits that $\sigma(L_V)$ contains a half line whenever $V$ is periodic and $d \geq 2$.  This conjecture inspired intense study, with contributions from many authors, including (but not limited to) \cite{HelMoh98,Karp97,PopSkr81,Skr79,Skr84,Skr85,Vel88}, and culminating in the paper of Parnovski \cite{Parn2008AHP}. For a more detailed discussion see also \cite{Kuchment2016BAMS} and references therein. In recent years, there has been renewed interest in the Bethe--Sommerfeld conjecture for other types of operators, including operators on quantum graphs \cite{ET2017} and discrete Schr\"odinger operators. Following a partial result in \cite{KrugerPreprint}, a discrete version of the Bethe--Sommerfeld conjecture was proved on $\Z^2$ \cite{EF18+}, on $\Z^d$ for general $d \geq 2$ \cite{HJ18}, and on other lattices \cite{FH18+}.

This shows that periodic spectra exhibit a marked difference in their topological structure as one passes from $d = 1$ to $d \ge 2$. It also suggests that for almost periodic potentials, and especially limit-periodic potentials, in dimensions at least two, one should not necessarily expect to find gaps in the spectrum at arbitrarily high energies. On the one hand, this is not a formal consequence of the known results in the periodic case as the bottom of the half line in the spectrum may run off to infinity as one moves through a sequence of periodic approximations to a limit-periodic/almost periodic $V$. On the other hand, there actually has been work on aperiodic almost-periodic potentials $V$ for which it could be shown that the spectrum contains a half line \cite{KL13, KS18+}. That is, a version of the Bethe--Sommerfeld conjecture has been established in the almost-periodic context, beyond the periodic case. Moreover, no example of a multi-dimensional almost periodic potential $V$ is known for which $\sigma(L_V)$ does not contain a half line.

This discussion prompts one to expect that the answer to the question above is that the results from \cite{DFL2017} likely do not have a multi-dimensional counterpart, and in particular there are likely no multi-dimensional almost periodic potentials $V$ for which $\sigma(L_V)$ has zero Lebesgue measure, let alone zero Hausdorff dimension.
\medskip

Alas, the correct answer to the question is that the results from \cite{DFL2017} do have higher-dimensional counterparts. We show in this short note how to use the examples of \cite{DFL2017} to construct $V \in \LP(\R^d)$ for $d \geq 2$ so that $\sigma(L_V)$ has zero lower box counting dimension. As a consequence, for these $V$, $\sigma(L_V)$ is a set of zero Hausdorff dimension, and in particular it is a generalized Cantor set (closed, with empty interior, and without any isolated points) of zero Lebesgue measure. Thus, when one transitions from the periodic to the almost-periodic setting, the property set forth in the Bethe--Sommerfeld conjecture may fail in the most spectacular fashion: not only does the spectrum have infinitely many gaps, but these gaps are dense and have full Lebesgue measure.

Let us emphasize that, to the best of our knowledge, these are the first examples of Schr\"odinger operators in $L^2(\R^d)$ with $d \ge 2$ (with \emph{any} potential, not just an almost periodic potential) that have a generalized Cantor set as spectrum. The additional statements about this Cantor set being very thin in the sense of standard fractal dimensions are an added bonus.

\medskip

Let us recall how one defines the box-counting dimensions (also called the Minkowski dimensions or the Minkowski-Bouligand dimensions) of a bounded\footnote{Boundedness is necessary to ensure that $S$ may be covered by finitely many $\varepsilon$-boxes.} set $S\subseteq \R$. Given $\varepsilon > 0$, let $N(\varepsilon) = N(\varepsilon;S)$ denote the minimal number of intervals of length $\varepsilon$ needed to cover $S$. The upper and lower box-counting dimensions of $S$ are then defined by
\[
\dim_{\rm B}^+(S)
=
\limsup_{\varepsilon\downarrow 0} \frac{\log N(\varepsilon;S)}{\log(\varepsilon^{-1})},
\quad
\dim_{\rm B}^-(S)
=
\liminf_{\varepsilon\downarrow 0} \frac{\log N(\varepsilon;S)}{\log(\varepsilon^{-1})}.
\]

We will say that (a potentially unbounded set) $S \subseteq \R$ has zero lower box counting dimension if
\[
\dim_{\rm B}^-(S \cap [-a,a])  = 0
\]
for all $a>0$.

It is well known that any $S \subseteq \R$ with zero lower box counting dimension must have zero Hausdorff dimension and hence zero Lebesgue measure.

\begin{theorem} \label{t:zeroMeasMultiD}
There are multi-dimensional limit-periodic $V$ such that $\sigma(L_V)$ is a generalized Cantor set of zero Lebesgue measure. In fact, there is a dense subset $\mathcal{B} \subseteq \LP(\R)$ with the property that $\sigma(L_V)$ is a generalized Cantor set of zero lower box counting dimension whenever $V$ is of the form
\begin{equation} \label{eq:sepPot}
V(x_1,\ldots,x_d)
=
\sum_{j=1}^d W(x_j)
\end{equation}
with $W \in \mathcal{B}$.\footnote{Clearly, if $W \in \LP(\R)$, then $V \in \LP(\R^d)$.}
\end{theorem}

The key technical result we prove here, however, is a one-dimensional result:

\begin{theorem} \label{t:zeroBox}
There is a dense subset $\mathcal{B} \subseteq \LP(\R)$ such that, for every $V \in \mathcal{B}$, $\sigma(L_V)$ has zero lower box counting dimension.
\end{theorem}

Indeed, Theorem~\ref{t:zeroBox} quickly implies Theorem~\ref{t:zeroMeasMultiD}:

\begin{proof}[Proof of Theorem~\ref{t:zeroMeasMultiD}]
For separable potentials $V$ as in \eqref{eq:sepPot}, one can express $\sigma(L_V)$ as a Minkowski sum of the 1D spectra:
\[
\sigma(L_V)
=
\sigma(L_{W}) + \cdots + \sigma(L_{W})
=
\set{\sum_{j=1}^d y_j : y_j \in \sigma(L_{W}) \text{ for each } j}.
\]
Thus, the conclusion of the theorem follows from Theorem~\ref{t:zeroBox} and standard arguments about Minkowski sums of fractal sets, e.g.\ Corollary~\ref{cor:boxSumBdBelow}.
\end{proof}

\begin{remark}
Let us make a few comments on Theorem~\ref{t:zeroMeasMultiD}.
\begin{enumerate}

\item The analysis of \cite{DFL2017} supplies the key input. Since Damanik--Fillman--Lukic were able to incorporate a coupling constant into their construction, one can also incorporate coupling constants into the present work. 
    That is to say, there is a dense family $\mathcal{B}$ so that $\sigma(L_V)$ is a Cantor set of zero Hausdorff dimension whenever
\[
V(x_1,\ldots,x_d)
=
 \sum_{j=1}^d \lambda_jW(x_j)
\]
with $W \in \mathcal{B}$ and $\lambda_1, \ldots, \lambda_d > 0$. Notice that this statement does not follow from the fact that $\dim_{\rm B}^-(\sigma(L_{\lambda W}))=0$ for all $\lambda>0$ automatically, since in general  sum of sets of zero lower box counting dimension does not have to have zero lower box counting dimension. Nevertheless, one can extract from \cite{DFL2017} that the set $\mathcal{B}$ can be constructed in such a way that for any $W\in  \mathcal{B}$ and any $\lambda_1, \ldots, \lambda_d > 0$ there exists a sequence $\varepsilon_n\to 0$ such that $\lim_{n\to \infty} \frac{\log N(\varepsilon_n;\sigma(L_{\lambda_j W}))}{\log(\varepsilon_n^{-1})}=0$ for each $j=1,\ldots, d$. And this is sufficient to show that the set
$$
\sigma(L_{\lambda_1 W}) + \cdots + \sigma(L_{\lambda_d W})
$$
has zero lower box counting dimension.

\item Our proof will address the dimensional statements, which in turn imply that the spectrum has empty interior. Since we claim that it is a generalized Cantor set, let us mention that the spectrum is always closed and the one-dimensional spectra (and hence their Minkowski sums) have no isolated points, both by well-known general principles.

\item It is not hard to construct $V$'s in Theorem~\ref{t:zeroBox} which are also Gordon potentials in the sense of \cite{G76}. Thus, it is possible to produce examples in Theorem~\ref{t:zeroMeasMultiD} and \ref{t:zeroBox} which have purely singular continuous spectrum. The absence of absolutely continuous spectrum is immediate from the zero Lebesgue measure property, and the absence of point spectrum follows from the Gordon lemma \cite{G76}. More specifically, the Gordon lemma yields the absence of eigenvalues, and hence the continuity of all spectral measures for the one-dimensional operators. From there, the spectral measure of any element of $L^2(\R^d)$ of the form $\phi(x) = \phi_1(x_1)\cdots\phi_d(x_d)$ (with $\phi_j \in L^2(\R)$) is the convolution of continuous measures, hence continuous. Since linear combinations of such functions are dense in $L^2(\R^d)$, it follows that $L_V$ has purely continuous spectrum.

\item By substituting the analysis of \cite{avila2009CMP} for \cite{DFL2017}, one can follow the outline of this paper and produce multidimensional limit-periodic discrete Schr\"odinger operators having spectra of zero lower box dimension (and hence zero Hausdorff dimension and zero Lebesgue measure). We focus on the continuum case here, but the changes to pass to the discrete setting are cosmetic.

\item It is interesting to compare Theorem~\ref{t:zeroMeasMultiD} to the results of \cite{KL13} in dimension two. From the construction in the proof of Theorem~\ref{t:zeroMeasMultiD}, one can choose $W$ to be of the form
\[
W(x) = \sum_{j=1}^\infty W_j(x),
\]
where $W_j(x)$ is $2^{j-1}$-periodic, and one may then write
\[
V(x_1,x_2) = \sum_{j=1}^\infty \underbrace{W_j(x_1) + W_j(x_2)}_{\equiv V_j(x_1,x_2)}.
\]
Clearly then, our examples must fail to satisfy the decay estimate
\[
\|V_j\|_\infty
\leq
\hat{C} e^{-2^{\eta j}},
\qquad
\eta > \eta_0 > 0
\]
from \cite{KL13}. However, it is unclear what the optimal decay rate in Theorem~\ref{t:zeroMeasMultiD} is. To be more specific, the optimal rate of decay of $\|V_j\|_\infty$ depends on the optimal quantitative dependence of $N_0$ on $\varepsilon$ in Lemma~\ref{l:smallspec}, and the proof of Lemma~\ref{l:smallspec} from \cite{DFL2017} does not yield useful quantitative bounds on $N_0$.
\end{enumerate}
\end{remark}
The remainder of the paper is concerned with proving Theorem~\ref{t:zeroBox}. For the reader's convenience, we also attach an appendix explaining how to derive the dimension result alluded to in the previous proof.

\section{Preparatory Work: The Spectrum in Finite Energy Windows} \label{sec:prep}

We will  need the following elementary estimate on the number of bands that one may observe in a finite energy window. This follows from standard asymptotics for the bands of periodic Schr\"odinger operators; for a detailed proof, see \cite{DFL2017}.

\begin{lemma} \label{l:bandcount}
If $V$ is continuous and $T$-periodic and $a>0$, then the interval $[-a,a]$ intersects no more than
\[
\frac{T}{\pi} \sqrt{a+\|V\|_\infty} + 1
\]
bands of $\sigma(L_V)$.
\end{lemma}

We will also use the following lemma from \cite{DFL2017}.

\begin{lemma} \label{l:smallspec}
Suppose $V \in C(\R)$ is $T$-periodic, $\varepsilon > 0$, and $a > 0$. There exists $N_0 = N_0(a,V,\varepsilon) \in \Z_+$ such that the following holds true. For any integer $N \geq N_0$, setting $\tilde T := NT$, there is an $\tilde T$-periodic potential $\widetilde V$ such that $\|V - \widetilde V\|_\infty < \varepsilon$, and one has the measure estimate
\[
\Leb\!\left(\sigma(L_{\widetilde V}) \cap [-a,a]\right)
\leq
\exp\left(-\widetilde T^{1/2}\right),
\]
where $\Leb$ denotes Lebesgue measure.
\end{lemma}

\section{Zero Lower Box Dimension}

\begin{proof}[Proof of Theorem~\ref{t:zeroBox}]

Fix a $T_0$-periodic potential $V_0 \in C(\R)$, and let $\varepsilon_0 > 0$ be given.  We will construct a sequence $(V_n)_{n=1}^\infty$ consisting of periodic potentials so that $V_\infty = \lim_n V_n$ satisfies $\|V_0 - V_\infty\|_\infty  < \varepsilon_0$ and $\sigma(L_{V_\infty})$ has zero lower box counting dimension. For the sake of notation, define $L_{n} = -\Delta +  V_n$ and $\Sigma_{n} = \sigma(L_{n})$ for $1 \leq n \leq \infty$.
\newline

Denote $a_n = 2^n$, and take $\varepsilon_1 = \varepsilon_0/2$. By Lemma~\ref{l:smallspec}, there exists a potential $V_1$ of period $T_1$, which is a multiple of $T_0$, such that $\|V_0 - V_1 \|_\infty < \varepsilon_1$ and
\[
\delta_1
:=
\Leb(\Sigma_{1} \cap [-a_1,a_1])
<
\exp\left(-T_1^{1/2}\right).
\]
Inductively, having constructed $V_{n-1}$, $\delta_{n-1}$, and $\varepsilon_{n-1}$, define
\begin{equation}\label{eq:nextepsdef}
\varepsilon_{n}
=
\min\left(\frac{\varepsilon_{n-1}}{2},
\frac{\delta_{n-1}}{4} \right).
\end{equation}
By Lemma~\ref{l:smallspec}, we may construct a multiple $T_n$ of $T_{n-1}$ and a $T_{n}$-periodic potential $V_{n}$ with $\|V_n - V_{n-1}\|_\infty < \varepsilon_{n}$  such that
\begin{equation} \label{eq:smallspec}
\delta_{n}
:=
 \Leb(\Sigma_{n} \cap [-a_n,a_n])
<
\exp\left(-T_{n}^{1/2}\right).
\end{equation}
By completeness, $V_\infty = \lim_{n\to\infty}V_n$ exists. By definition, $V_\infty$ is limit-periodic. By \eqref{eq:nextepsdef}, we have
\[
\| V_0 - V_\infty\|_\infty
<
\sum_{j = 1}^\infty \varepsilon_j
\leq
\sum_{j=1}^\infty 2^{-j} \varepsilon_0
=
 \varepsilon_0.
\]

Thus, it remains to show that the spectrum has lower box dimension zero. Notice that \eqref{eq:nextepsdef} yields
\begin{equation} \label{eq:taildist}
\| V_n -  V_\infty\|_\infty
\leq
 \sum_{j = n + 1}^\infty \varepsilon_j
<
\sum_{k = 2}^\infty 2^{-k} \delta_n
=
\delta_n/2
\end{equation}
for all $n \in \Z_+$. We claim that
\begin{equation} \label{eq:claim}
\dim_{\mathrm B}^-(\Sigma_{\infty} \cap [-a_j,a_j]) = 0
\text{ for every } j \in \Z_+.
\end{equation}

To see this, let $n \in \Z_+$ with $n \geq j$ be given. Then, by \eqref{eq:taildist}, the $\delta_n/2$-neighborhood of $\Sigma_{n} \cap [-a_j,a_j]$ together with the intervals $[-a_j,-a_j + 2\delta_n]$ and $[a_j - 2\delta_n,a_j]$ comprises a cover of $\Sigma_{\infty} \cap [-a_j,a_j]$ by intervals of length at most $2\delta_n$. By \eqref{eq:smallspec} and Lemma~\ref{l:bandcount}, we have
\[
\frac{\log N(2\delta_n;\sigma(L_{V_\infty}) \cap [-a_j,a_j])}{\log ((2\delta_n)^{-1})}
\lesssim
\frac{\log \left(\frac{1}{\pi} T_n \sqrt{\|V_0\|_\infty + \varepsilon_0 + a_j} + 3 \right) }{\sqrt{T_n}},
\]
which clearly goes to zero as $n \to \infty$, proving \eqref{eq:claim}. Having proved \eqref{eq:claim}, we are done.
\end{proof}

\section*{Acknowledgment} We are grateful to Leonid Parnovski for useful comments on an earlier version of this manuscript.

\begin{appendix}

\section{Minkowski Self-Sums of a Set with Lower Box Dimension Zero}

\begin{prop} \label{prop:boxsum}
Let $C$ be a bounded set, and, for $d \geq 1$, denote
\[
C^{(d)}
=
\underbrace{C+ \cdots +C}_{d \textup{ copies of  } C}.
\]
If $\dim_{\rm B}^-(C) = 0$, then $\dim_B^-(C^{(d)})=0$ for all $d \geq 1$. In particular, $C^{(d)}$ has zero Hausdorff dimension and hence zero Lebesgue measure.
\end{prop}

\begin{proof}
Let $\varepsilon_j \downarrow 0$ be such that $\log N(\varepsilon_j;C) / \log( \varepsilon_j^{-1}) \to 0$, and let $\delta > 0$ be given. Then, for large $j$, $C$ may be covered by fewer than $\varepsilon_j^{-\delta}$ intervals of length at most $\varepsilon_j$. Consequently, $C^{(d)}$ may be covered by fewer than $\varepsilon_j^{-\delta d}$ intervals of length at most $d \varepsilon_j$. We then have
\[
\frac{\log N(d\varepsilon_j; C^{(d)})}{\log((d\varepsilon_j)^{-1})}
\leq
\frac{\delta d \log(\varepsilon_j^{-1})}{\log(1/d) + \log(\varepsilon_j^{-1})}.
\]
Sending $j \to \infty$, we get $\dim_{\rm B}^-(C^{(d)}) \leq \delta d$. Sending $\delta \downarrow 0$ concludes the argument.
\end{proof}

\begin{coro} \label{cor:boxSumBdBelow}
If $C \subseteq \R$ is bounded from below and has lower box dimension zero, then $C^{(d)}$ has lower box dimension zero for all $d \geq 1$.
\end{coro}

\begin{proof}
This follows immediately from Proposition~\ref{prop:boxsum} and the following observation: if $C$ is bounded from below by $- \gamma \leq 0$, then, for any $d\geq 1$ and $a > 0$, one has
\[
C^{(d)} \cap [-a,a]
\subseteq
\left(C \cap [-\gamma,a+(d-1)\gamma]\right)^{(d)}.
\]
\end{proof}
Notice that the assumption that $C$ is bounded below is crucial. To see this, let $\alpha$ be irrational and consider
\[
C = \Z_+ \cup (-\alpha \Z_+).
\]
Clearly, $C$ has box dimension zero, but $C+C$ does not.

\end{appendix}

\end{document}